\documentclass[11pt, reqno]{article}
\usepackage{amsmath,amssymb,amsfonts,amsthm}
\usepackage{color}

\usepackage[colorlinks=true,linkcolor=blue,citecolor=blue,urlcolor=blue,pdfborder={0 0 0}]{hyperref}
\usepackage{cleveref}

\usepackage{caption}
\usepackage{thmtools}

\usepackage{mathrsfs}

\crefname{theorem}{Theorem}{Theorems}
\crefname{thm}{Theorem}{Theorems}
\crefname{lemma}{Lemma}{Lemmas}
\crefname{lem}{Lemma}{Lemmas}
\crefname{remark}{Remark}{Remarks}
\crefname{prop}{Proposition}{Propositions}
\crefname{defn}{Definition}{Definitions}
\crefname{corollary}{Corollary}{Corollaries}
\crefname{conjecture}{Conjecture}{Conjectures}
\crefname{question}{Question}{Questions}
\crefname{chapter}{Chapter}{Chapters}
\crefname{section}{Section}{Sections}
\crefname{figure}{Figure}{Figures}

\theoremstyle{plain}
\newtheorem{thm}{Theorem}[section]

\newtheorem{lemma}[thm]{Lemma}
\newtheorem{theorem}[thm]{Theorem}

\newtheorem{corollary}[thm]{Corollary}

\newtheorem{prop}[thm]{Proposition}
\newtheorem{conjecture}[thm]{Conjecture}

\theoremstyle{definition}

\newtheorem{problem}[thm]{Problem}

\theoremstyle{remark}
\newtheorem{remark}[thm]{Remark}

\numberwithin{equation}{section}

\renewcommand{\P}{\mathbb P}
\newcommand{\E}{\mathbb E}

\newcommand{\R}{\mathbb R}
\newcommand{\Z}{\mathbb Z}
\newcommand{\N}{\mathbb N}

\newcommand{\cE}{\mathcal E}

\newcommand{\cO}{\mathcal O}

\usepackage[margin=1.0225in]{geometry}
\usepackage{extarrows}
\usepackage{titling}
\usepackage{commath}
\usepackage{mathtools}
\usepackage{titlesec}
\newcommand{\eps}{\varepsilon}
\usepackage{bbm}
\usepackage{setspace}
\usepackage{enumitem}

\usepackage[textsize=tiny]{todonotes}

\usepackage{cite}

\newcommand{\Aut}{\operatorname{Aut}}

\newcommand{\bP}{\mathbf P}
\newcommand{\bE}{\mathbf E}

\newcommand{\stab}{\operatorname{Stab}}

\def\P{\mathbb{P}}
\renewcommand{\Pr}{ \mathrm P}

\newcommand{\la}{\lambda}

\DeclareMathSymbol{\leqslant}{\mathalpha}{AMSa}{"36} % nicer `smaller or equal'
\DeclareMathSymbol{\geqslant}{\mathalpha}{AMSa}{"3E} % nicer `larger or equal'
\DeclareMathSymbol{\eset}{\mathalpha}{AMSb}{"3F}     % nicer `emptyset'
\renewcommand{\epsilon}{\varepsilon}

\newcommand{\EE}{\mathcal{E}}

\newcommand{\sfrac}[2]{\mbox{\small $\frac{#1}{#2}$}}
\newcommand{\ssfrac}[2]{\mbox{\footnotesize $\frac{#1}{#2}$}}
\newcommand{\half}{\ssfrac{1}{2}}

\pretitle{\begin{flushleft}\Large}
\posttitle{\par\end{flushleft}\vskip 0.5em}
\preauthor{\begin{flushleft}}
\postauthor{
\\
\vspace{0.5em}
\footnotesize{Statslab, DPMMS, University of Cambridge.\\ Email: \href{mailto:jh2129@statslab.cam.ac.uk}{jh2129@statslab.cam.ac.uk} and
\href{mailto:t.hutchcroft@maths.cam.ac.uk}{t.hutchcroft@maths.cam.ac.uk}}
\end{flushleft}}
\predate{\begin{flushleft}}
\postdate{\par\end{flushleft}}
\title{\bf No 
percolation at criticality 
on certain groups of intermediate growth
}
\renewenvironment{abstract}
 {\par\noindent\textbf{\abstractname.}\ \ignorespaces}
 {\par\medskip}

\titleformat{\section}
  {\normalfont\large \bf}{\thesection}{0.4em}{}

\author{{\bf Jonathan Hermon and Tom Hutchcroft}}
\begin{document}

\date{\small{\today}}

\maketitle

\setstretch{1.1}

\begin{abstract}
We prove that critical percolation has no infinite clusters almost surely on any unimodular quasi-transitive graph satisfying a return probability upper bound of the form $p_n(v,v) \leq \exp\left[-\Omega(n^\gamma)\right]$ for some $\gamma>1/2$. The result is new in the case that the graph is of intermediate volume growth.
% , and applies for example to Cayley graphs of Erschler's piecewise automatic groups.
\end{abstract}

\section{Introduction}\label{sec:intro}

In \textbf{Bernoulli bond percolation}, first studied by Broadbent and Hammersley \cite{MR91567}, each edge of a connected, locally finite graph $G$  is either deleted or retained at random  with retention probability $p\in[0,1]$, independently of all other edges. We denote the random graph obtained this way by $\omega_p$. Connected components of $\omega_p$ are referred to as \textbf{clusters}.
 Percolation theorists are primarily interested in the geometry of the open clusters and how this geometry changes as the parameter $p$ is varied. We are particularly interested in \emph{phase transitions}, where the geometry of $\omega_p$ changes abruptly as we vary $p$ through some special value.  
% Given a graph $G$, the \textbf{critical probability} is defined to be 
% \[p_c= \sup \left\{p\in [0,1] : G[p] \text{ has no infinite clusters almost surely}\right\}, \]
% while the \textbf{uniqueness threshold} is defined to be
% \[p_u= \inf \left\{p \in [0,1] : G[p] \text{ has a unique infinite cluster almost surely}\right\}.\]
% Percolation was first studied by Broadbent and Hammersley \cite{MR91567}; 
% 
% 
% 
% These questions were first studied by Broadbent and Hammersley \cite{MR91567}; we refer the reader to e.g. \cite{grimmett2010percolation,1707.00520,heydenreich2015progress} for further background.
% 
% 
The first basic result about percolation, without which the model would not be nearly as interesting, is that for most infinite graphs (excluding e.g.\ one-dimensional counterexamples such as the infinite line graph $\Z$), percolation undergoes a \emph{non-trivial phase transition}, meaning that the \textbf{critical probability}
\[p_c(G) = \inf\bigl\{ p \in [0,1] : \omega_p \text{ has an infinite cluster almost surely}\bigr\}\]
is strictly between zero and one\footnote{It is a consequence of Kolmogorov's 0-1 law that the probability that $\omega_p$ contains an infinite cluster is either $0$ or $1$. Moreover, straightforward coupling arguments show that if $\omega_p$ contains an infinite cluster almost surely then $\omega_q$ contains an infinite cluster almost surley for every $p \leq q \leq 1$, see \cite[Page 11]{grimmett2010percolation}.}. Indeed, a very general result to this effect has recently been proven by Duminil-Copin, Goswami, Raoufi, Severo, and Yadin \cite{1806.07733}, which implies in particular that $0<p_c<1$ for every quasi-transitive graph of superlinear volume growth.
% (Conjecturally, $p_c \in (0,1)$ for every transitive graph not rough-isometric to $\Z$.
 % This has been resolved in most cases, and remains open only for transitive graphs of intermediate volume growth i.e., subexponential but superpolynomial.)

Once we know that the phase transition is non-trivial, the next question is to determine what happens when $p$ is exactly equal to the critical value $p_c$. This is a much more delicate question. Indeed, one of the most important open problems in percolation theory is to prove that critical percolation on the $d$-dimensional hypercubic lattice $\Z^d$ does not contain any infinite clusters almost surely for every $d\geq 2$. This problem was solved in two dimensions by Russo in 1981 \cite{russo1981critical}, and for all $d\geq 19$ by Hara and Slade in 1994 \cite{hara1994mean}. More recently, Fitzner and van der Hoftstad \cite{fitzner2015nearest} sharpened the methods of Hara and Slade to solve the problem for all $d\geq 11$. It is expected that this method can in principle, and with great effort and ingenuity, be pushed to handle all $d\geq 7$, while dimensions $3,4,5,$ and $6$ are expected to require new approaches. 
Similar results for other Euclidean lattices have been obtained in \cite{MR1124831,MR1144091,MR3503025}.

% Percolation was originally studied mostly in the setting of Euclidean lattices.
In their highly influential paper \cite{bperc96},  Benjamini and Schramm proposed  a systematic study of percolation on general \textbf{transitive} graphs, that is, graphs for which the action of the automorphism group on the vertex set has a single orbit (i.e., graphs for which any vertex can be mapped to any other vertex by a symmetry of the graph), and more generally on \textbf{quasi-transitive graphs}, for which there are only finitely many orbits. Prominent examples of transitive graphs include Cayley graphs of finitely generated groups. The following is among the most important of the many outstanding conjectures that they formulated.

% Let $G=(V,E)$ be a connected, locally finite graph. For each $p\in [0,1]$ we define $\omega_p \in \{0,1\}^E$ to be the a random subgraph of $G$ obtained by either deleting or retaining each edge of $G$ independently at random, with retention probability $p$. This is known as \textbf{Bernoulli bond percolation} on $G$. The connected components of $\omega_p$ are known as \textbf{clusters}. In the theory of percolation, we are especially interested in the study of percolation at and near the \textbf{critical probability}
% \[
% p_c(G)=\inf\bigl\{p\in [0,1] : \omega_p \text{ has an infinite cluster almost surely}\bigr\}.
% \]
% In particular, one of the most important questions in percolation asks whether or not there are any infinite clusters at the critical probability. Historically, these questions were studied primarily on Euclidean lattices such as the hypercubic lattice $\Z^d$. 
% One of the most important questions in percolation is determine whether or not 

\begin{conjecture}[Benjamini and Schramm 1996]
\label{conj:pc}
Let $G$ be a quasi-transitive graph. If $p_c(G)<1$ then critical Bernoulli bond percolation on $G$ has no infinite clusters almost surely.
\end{conjecture}

Aside from the previously mentioned results in the Euclidean setting, previous progress on \cref{conj:pc} can briefly be summarised as follows. 
% The research program Benjamini and Schramm initiated saw a huge amount of progress in the late 1990's and early 2000's, particularly in percolation on \emph{unimodular, nonamenable} transitive graphs, for which 
Benjamini, Lyons, Peres, and Schramm \cite{BLPS99b} proved that \cref{conj:pc} holds  for every \emph{unimodular, nonamenable} transitive graph. Here, \emph{unimodularity} is a technical condition that holds for every Cayley graph and every amenable quasi-transitive graph; see \cref{sec:unimodular} for further background. Tim\'ar \cite{timar2006percolation} later showed that critical percolation on any \emph{nonunimodular} transitive graph cannot have \emph{infinitely many} infinite clusters. Both results are easily generalised to the quasi-transitive setting. In \cite{Hutchcroft2016944}, the second author of this article showed that critical percolation on any quasi-transitive graph of \emph{exponential growth} cannot have a \emph{unique} infinite cluster. Together with the aforementioned results of Benjamini, Lyons, Peres, and Schramm and Tim\'ar, this established that \cref{conj:pc} holds for every quasi-transitive graph of exponential growth. An alternative proof of this result in the unimodular case was recently given in \cite{1808.08940}. 
% This proof is very specific to the exponential growth setting, however, and does not yield any non-trivial information at all under weaker  assumptions.
All of these proofs have elements that are very specific to the exponential growth setting, and completely break down without this assumption. 

In this paper, we build upon the ideas of \cite{1808.08940} to develop a new method of proving that there are no infinite clusters at criticality. % to develop a new method of proving prove a new criterion for there to be no infinite clusters at criticality.
 % In addition to various examples for which the result was already known,
  This new method applies in particular to certain transitive graphs of \emph{intermediate growth}, for which the volume $|B(v,r)|$ of a ball of radius $r$ grows faster than any polynomial in $r$ but slower than any exponential of $r$. (In notation\footnote{Here we use Landau's asymptotic notation: In particular, for non-negative $f(n)$ and $g(n)$, ``$f(n)=o(g(n))$ as $n\to \infty$'' and ``$g(n)=\omega(f(n))$ as $n\to\infty$'' both  mean that $\lim_{n\to\infty} f(n)/g(n) = 0$, while ``$f(n) = O(g(n))$ as $n \to \infty$'' and ``$g(n)=\Omega(f(n))$ as $n \to\infty$'' both mean that $\limsup_{n\to\infty} f(n)/g(n) < \infty$.}, a graph has intermediate growth if $r^{\omega(1)}\leq |B(v,r)| \leq e^{o(r)}$ as $r\to \infty$.) No such graph had previously been proven to satisfy \cref{conj:pc}.  The hypotheses of our results are most easily stated in terms of the $n$-step simple random walk return probabilities $p_n(v,v)$. Given $c>0$ and $0<\gamma\leq 1$, we say that a graph satisfies \eqref{ass:HK} if
\begin{equation}
\label{ass:HK}
p_n(v,v) \leq  \exp\left[ -c n ^{\gamma} \right] \qquad \text{ for every $v\in V$ and $n\geq 1$.} \tag{$\mathrm{HK}_{\gamma,c}$}
\end{equation}
We can now state our main theorem. 

\begin{theorem}
\label{thm:main}
Let $G$ be a unimodular quasi-transitive graph satisfying \eqref{ass:HK} for some $c>0$ and $\gamma >1/2$. Then critical Bernoulli bond percolation on $G$ has no infinite clusters almost surely.
\end{theorem}

See \cref{sec:closing} for a discussion of some variations on this result and a discussion of how our proof breaks down in the case $\gamma <1/2$.
Examples of groups of intermediate growth whose Cayley graphs satisfy the hypotheses of \cref{thm:main} can be constructed as \emph{piecewise automatic groups} \cite[Corollary 1]{MR2254627} or using the notion of  \emph{diagonal products} \cite{MR3034295}. (An analysis of the heat kernel on diagonal products will appear in a forthcoming work of Amir and Zheng.) Further examples can easily be constructed by, say, taking products of these groups with other groups of subexponential growth.   For further background on groups of intermediate growth see \cite{MR2478087} and references therein.
Further works concerning probability on groups of intermediate growth include \cite{MR1841989,MR3607808,erschler2018growth}. 

% Unfortunately, the technical nature of our proof makes it difficult to see clearly why it breaks down in the case $\gamma<1/2$. 

Note that \cref{thm:main} also implies that $p_c<1$, so that we obtain an independent proof of the recent result of \cite{1806.07733} in the special case of the class of graphs we consider. We also remark that \cref{thm:main} implies that there is no percolation at $p_c$ on any unimodular quasi-transitive graph satisfying an isoperimetric inequality of the form $|\partial K|\geq c|K|/\log^\delta |K|$ for $c>0$ and $0<\delta<1/2$, see \cite{MR2198701} and \cref{remark:isoperimetry}.
% include Cayley graphs of Erschler's piecewise automatic groups 

The proof of \cref{thm:main} is quantitative, and also yields explicit  bounds on the tail of the volume of a critical cluster. In particular, we obtain the following bound in the transitive setting. The corresponding bound for quasi-transitive graphs is given in \cref{thm:quantquasi}. We write $\bP_p$ and $\bE_p$ for probabilities and expectations taken with respect to the law of $\omega_p$ and write $K_v$ for the cluster of $v$ in $\omega_p$. 

\begin{theorem}
\label{thm:quanttransitive}
Let $G=(V,E)$ be a unimodular transitive graph with maximum degree at most $M$ satisfying \eqref{ass:HK} for some $c>0$ and $\gamma >1/2$. Then for every $0\leq \beta < (2\gamma-1)/\gamma$ there exists $C(\beta)=C(\beta,\gamma,M,c)$ such that
\[
\bE_{p}\exp\left[\log^\beta |K_v|\right] \leq C(\beta)\]
for every $p\leq p_c$.
\end{theorem}

We expect these bounds to be very far from optimal. Indeed, it is widely believed that critical percolation on any quasi-transitive graph of at least seven dimensional volume growth should satisfy $\bP_{p_c}(|K_v| \geq n) \preceq n^{-1/2}$ as $n\to\infty$. See e.g.\ \cite{Hutchcroftnonunimodularperc,1804.10191,heydenreich2015progress} and references therein for a detailed discussion of what is currently known regarding such bounds.

% \medskip

An immediate corollary of \cref{thm:quanttransitive} is that Schramm's locality conjecture \cite[Conjecture 1.2]{MR2773031} holds in the case of graph sequences uniformly satisfying \eqref{ass:HK} for some $\gamma>1/2$. 

\begin{corollary}
\label{cor:locality}
Let $(G_n)_{n\geq 1}$ be a sequence of infinite unimodular transitive graphs converging locally to a transitive graph $G$, and suppose that there exist $c>0$ and $\gamma>1/2$ such that $G_n$ satisfies \eqref{ass:HK} for every $n\geq 1$. Then $p_c(G_n) \to p_c(G)$ as $n\to\infty$.
\end{corollary}

See \cite{1808.08940,MR2773031} for a detailed discussion of this conjecture and for the definition of local convergence of graphs. 
The proof of \cref{cor:locality} given \cref{thm:quanttransitive} is very similar to the proof of \cite[Corollary 5.1]{1808.08940} and is omitted.

\paragraph{Proof overview}
The proof of \cref{thm:main,thm:quanttransitive} applies several of the ideas developed in the second author's recent paper \cite{1808.08940}, which we now review. Briefly, the methods of that paper allow us to convert bounds on the \textbf{two-point function} $\tau_p(u,v)$, defined to be the probability that $u$ and $v$ are connected in $\omega_p$, into bounds on the tail of the volume of a cluster whenever $0<p<p_c$. This is done as follows. For each set $K \subseteq V$, we write $E(K)$ for the set of edges of $G$ that \textbf{touch} $K$, i.e., have at least one endpoint in $K$. For each edge $e$ of $G$ and $n\geq 1$, let $\mathscr{S}_{e,n}$ be the event that $e$ is closed and that the endpoints of $e$ are in distinct clusters each of which touches at least $n$ edges and at least one of which is finite. The following universal inequality is proven in \cite{1808.08940} using a variation on the methods of Aizenman, Kesten, and Newman \cite{MR901151}. It is a form of what we call the \emph{two-ghost inequality}.

\begin{theorem}
\label{thm:twoghosttransitive}
Let $G=(V,E)$ be a unimodular transitive graph of degree $d$. Then
\[\bP_p(\mathscr{S}_{e,n}) \leq
66 d \left[\frac{1-p}{pn}\right]^{1/2}
\]
for every $e\in E$, $p\in [0,1]$ and $n \geq 1$.
\end{theorem}

Next, an insertion-tolerance argument \cite[Equation 4.2]{1808.08940} is used to bound the tail of the volume in terms of the two-point function and the probability of $\mathscr{S}_{e,n}$ as follows. We define $\kappa_p(k) = \inf\{\tau_p(u,v) : u,v\in V, d(u,v)\leq k\}$, where $d(u,v)$ denotes the graph distance, and define $P_p(n) = \inf_{v\in V}\bP_p(|E(K_v)| \geq n)$.

\begin{lemma}
\label{lem:surgery}
 Let $G$ be a connected, locally finite graph. Then
\[
P_p(n)^2 \leq \kappa_p(k) + \left[ \sum_{i=0}^{k-1} p^{-i}\right] \sup_{e\in E} \bP_p(\mathscr{S}_{e,n})
\]
for every $0\leq p <p_c$, $n\geq 1$ and $k\geq 1$.  
\end{lemma}

Combining \cref{thm:twoghosttransitive,lem:surgery} allows us to convert bounds on $\kappa_p(k)$ into bounds on $P_p(n)$ when $G$ is transitive and unimodular. For graphs of exponential growth, this was enough to conclude a bound of the form $P_{p_c}(n)\preceq n^{-\delta}$ using the exponential two-point function bound $\kappa_{p_c}(k)\leq \operatorname{gr}(G)^{-k}$ that was proven in \cite{Hutchcroft2016944}.

In our setting, however, we do not have any non-trivial \emph{a priori} control of the rate of  decay of $\kappa_{p_c}(k)$. (Indeed, if we had such control we would already know that there is no percolation at $p_c$!) We circumvent this issue using the following bootstrapping procedure. We first prove via classical random walk techniques that if a transitive graph satisfies \eqref{ass:HK} for some $c>0$ and $0<\gamma \leq 1$ then there exists $c'>0$ such that the estimate
\[
\Pr_{\mu_A}(X_k \in A) \leq \exp\left[-c' \min\left\{k^\gamma,\, \frac{k}{\log^\alpha |A|}\right\}
\right]
\]
holds for every finite set $A\subset V$ and $k\geq 0$, where $\alpha = (1-\gamma)/\gamma$ and $\Pr_{\mu_A}$ denotes the law of the random walk $(X_k)_{k\geq 0}$ started from a uniformly random vertex of $A$. This is done in \cref{sec:randomwalk}. Taking expectations, this gives in the transitive unimodular case that
\[
\kappa_p(k) \leq \bE_p\left[\Pr_{\rho}(X_k \in K_\rho)\right] = \bE_p\left[\Pr_{\mu_{K_\rho}}(X_k \in K_\rho)\right] \leq \bE_p \exp\left[-c' \min\left\{k^\gamma,\, \frac{k}{\log^\alpha |K_\rho|}\right\}\right],
\]
where the central equality follows from the mass-transport principle. 
Thus, we now have methods both for converting bounds on $\kappa_p$ into bounds on $P_p$ and vice versa, so that in particular we can convert one bound on $P_p(n)$ into another via an intermediate bound on $\kappa_p(k)$.

%%%%%%%%%%%%%%%%%%%%%%%%%%

% \begin{figure}
% \centering
% \smartdiagram[flow diagram:horizontal]{
%   Assume bound on $P_p(n)$, Obtain bound on $\kappa_p(k)$, Obtain new bound on $P_p(n)$}
% \end{figure}
\usetikzlibrary{shapes,arrows}

% Define block styles
\tikzstyle{decision} = [diamond, draw, fill=blue!20, 
    text width=4.5em, text badly centered, node distance=3cm, inner sep=0pt]
\tikzstyle{block} = [rectangle, draw, fill=blue!20, 
    text width=5em, text centered, rounded corners, minimum height=4em]
\tikzstyle{line} = [draw, -latex']
\tikzstyle{cloud} = [draw, ellipse,fill=red!20, node distance=3cm,
    minimum height=2em]

\begin{figure}
\centering
\includegraphics[width=\textwidth]{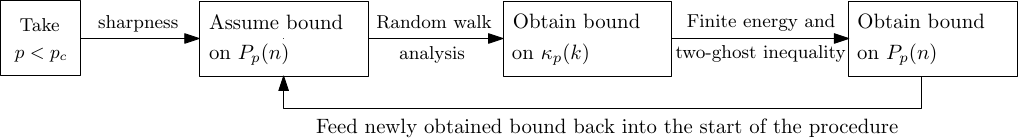}
\caption{Schematic illustration of the bootstrapping procedure used (implicitly) in the proof of \cref{thm:quanttransitive}.}
\label{fig:bootstrap}
\end{figure}

%%%%%%%%%%%%%%%%%%%%%%%%%%

On the other hand, we know by sharpness of the phase transition \cite{MR852458,aizenman1987sharpness,duminil2015new} that $\bE_p|K_\rho|<\infty$ for every $0\leq p < p_c$, and consequently that for each $0\leq p < p_c$ there exists a constant $C_p$ such that $P_p(n)\leq C_p n^{-1}$ for every $n\geq 1$. 
To conclude the proof, it suffices to show that if we start with this bound and iteratively obtain new bounds on $P_p(n)$ using the above method, then in the case $\gamma>1/2$ we obtain in the limit a bound on $P_p(n)$ that decays as $n\to\infty$ and holds uniformly on the whole range $0 \leq p < p_c$, as the same bound must then hold at $p_c$ by an elementary continuity argument. See \cref{fig:bootstrap} for a  schematic outline. (A discussion of how this proof strategy breaks down in the case $\gamma < 1/2$ is given in \cref{remark:obstacle}.) Rather than carrying out such a procedure explicitly, we instead use a similar method to prove a bound of the form
\[
\bE_p\exp\left[\log^\beta |K_\rho|\right] \leq C_\beta \sqrt{\bE_p\exp\left[\log^\beta |K_\rho|\right]}
\]
 for each $p_c/2 \leq p < p_c$ and $0\leq  \beta < (2\gamma-1)/\gamma$, which conveniently encapsulates this bootstrapping scheme and easily allows us to conclude the proof.

\section{Background on unimodularity and the mass-transport principle}
\label{sec:unimodular}

We now briefly review the notion of unimodularity and the mass-transport principle, referring the reader to \cite[Chapter 8]{LP:book} for further background. Let $G=(V,E)$ be a connected, locally finite graph and let $\Aut(G)$ be the group of automorphisms of $G$. We write $[v]=\{\gamma v : v \in \Aut(G)\}$ for the orbit of a vertex $v\in V$ under $\Aut(G)$ and say that $G$ is \textbf{unimodular} if $|\stab_u v|=|\stab_v u|$ for every $u,v\in V$ with $[u]=[v]$, where $\stab_u = \{\gamma \in \Aut(G) : \gamma u = u\}$ is the stabilizer of $u$ in $\Aut(G)$ and $\stab_u v = \{ \gamma v : \gamma \in \stab_u\}$ is the orbit of $v$ under $\stab_u$. Every Cayley graph and every amenable quasi-transitive graph is unimodular \cite{MR1082868}. 

Suppose that $G$ is a connected, locally finite, transitive unimodular graph. Then $G$ satisfies the \textbf{mass-transport principle}, which states that for every $F:V^2\to [0,\infty]$ that is diagonally-invariant in the sense that $F(\gamma u, \gamma v) = F(u,v)$ for every $u,v\in V$ and $\gamma \in \Aut(G)$, we have that
\[
\sum_{v\in V} F(\rho,v) = \sum_{v\in V} F(v,\rho)
\]
whenever $\rho$ is an arbitrarily chosen root vertex of $G$. More generally, suppose that $G$ is a connected, locally finite, \emph{quasi-transitive} unimodular graph, and let $\cO \subseteq V$ be a set of orbit representatives of the action of $\Aut(G)$. That is, $\cO$ is such that for every $v\in V$ there exists a unique $o\in \cO$ such that $[v]=[o]$. Then there exists a unique probability measure $\mu$ on $\cO$ such that the identity
\[
\sum_{o\in \cO}\sum_{v\in V} F(o,v) \mu(o) = \sum_{o\in \cO}\sum_{v\in V} F(v,o) \mu(o) 
\]
holds for every diagonally invariant $F:V^2\to[0,\infty]$. In other words, if we choose a root $\rho\in V$ according to the measure $\mu$ then $(G,\rho)$ is a \emph{unimodular random rooted graph} in the sense of \cite{AL07}. 
Similarly, if we choose $\rho$ according to the degree-biased probability measure defined by
\[\tilde \mu(o) = \frac{\mu(o)\deg(o)}{\sum_{o'\in\cO} \mu(o') \deg(o')} \qquad o\in \cO\]
then the random $(G,\rho)$ is a \emph{reversible random rooted graph} in the sense of \cite{BC2011} (we will not make substantial use of these notions so we omit the definition). This gives rise to the following generalization of the two-ghost inequality to the quasi-transitive case, see \cite[Remark 6.1]{1808.08940}.

\begin{theorem}
\label{thm:twoghostquasitransitive}
Let $G=(V,E)$ be a connected, locally finite, unimodular quasi-transitive graph. Then
\[
\sum_{o\in \cO}\mu(o) \sum_{e^- = o}\bP_p(\mathscr{S}_{e,n}) \leq
66 \left[\sum_{o\in \cO} \mu(o) \deg(o) \right] \left[\frac{1-p}{pn}\right]^{1/2}
\]
for every $p\in [0,1]$ and $n \geq 1$.
\end{theorem}

\section{Random walk analysis}
\label{sec:randomwalk}

The goal of this section is to prove the following inequality regarding simple random walk on graphs satisfying \eqref{ass:HK}, which will play an important role in the proof of our main theorems. 
Given a locally finite graph $G=(V,E)$ and a finite set $D \subseteq V$, we write $\mu_D$ for the uniform measure on $D$. For each probability measure $\mu$ on $V$, we also write $\Pr_{\mu}$ and $\mathrm{E}_{\mu}$ for probabilities and expectations taken with respect to the law of a simple random walk $(X_k)_{k\geq 0}$ started at a vertex drawn from the measure $\mu$. 

\begin{prop}
\label{cor:SP}
Let $G=(V,E)$ be an infinite, connected graph with degrees bounded by  $M$ satisfying \eqref{ass:HK} for some $c>0$ and $0<\gamma \leq 1$, and let $\alpha=(1-\gamma)/\gamma$. Then there exists a positive constant $c_1=c_1(\gamma,c,M)$ such that
\begin{equation}
\label{e:eSPcor}
\Pr_{\mu_D}\bigl(X_k \in D \bigr) \leq   \left[\max_{u,v \in D } \frac{\deg(u)}{\deg(v)}\right]^{1/2} \exp\left[ - c_1 \min\left\{\frac{k}{\log^\alpha |D|},\, k^{\gamma}\right\} \right]\end{equation}
 for every finite set $D\subset V$ and every $k\geq 0$.
\end{prop}

We expect that much of the content of this section will have been known as folklore by experts in random walks, but \cref{cor:SP} has not, to our knowledge,  previously appeared in the literature. Indeed, \cref{cor:SP} will be deduced from a more general estimate, \cref{cor:spcor}, which is a direct analogue in the infinite-volume setting of the $L^{\infty}$ mixing time bounds of Goel, Montenegro, and Tetali \cite{MR2199053}. 

The proof of \cref{cor:SP} will apply the notion of the \emph{spectral profile}, which we now introduce.
Let $G=(V,E)$ be an infinite, connected, locally finite graph, and let $P$ be the transition matrix of the simple random walk  
% (\textbf{SRW})
 $(X_k)_{k=0}^{\infty}$ on $G$. 
% ,  given by $P(a,b)=\sfrac{\mathbbm{1}(\{a,b\} \in E )}{\deg a} $.
For each finite set $A \subset V $ we define $P_A$ to be the substochastic transition matrix of the  random walk that is killed upon exiting $A$, which is given explicitly by $P_A(u,v)=P(u,v)\mathbbm{1}(u,v\in A)$, and define $\la(A) $ to be the smallest eigenvalue of $I_A-P_A^{2}$, where $I_A(u,v)=\mathbbm{1}(u=v, u \in A)$ and where we write $P_A^i$ for $(P_A)^{i}$. Let $\pi $ be the measure on $V$ which assigns each $v$ mass $\deg v$.
% Note that, by the Perron-Frobenius Theorem, we have that 
% \[
% \max\{|\lambda| : \lambda \text{ an eigenvalue of $P_A$}\}=\max\{\lambda : \lambda \text{ an eigenvalue of $P_A$}\},
% \]
% for every $A \subseteq V$ finite, from which it follows by elementary calculation that
% \begin{equation}
% \frac{1}{2}\lambda(A) \leq 1-\max\bigl\{|\lambda| : \lambda \text{ an eigenvalue of $P_A$}\bigr\} = 1- \sqrt{1-\lambda(A)} 
% % \frac{1}{2}\left(\sqrt{4+\lambda(A)^2}-\lambda(A)\right)
%  \leq \lambda(A)
% \label{eq:removingthesquare}
% \end{equation}
% for every $A \subseteq V$ finite. 
 We define the \textbf{spectral profile} of $G$ to be the function $\Lambda: \N \to (0,1]$ given by
 \[ \Lambda(L):=\inf \bigl\{\la(B):B \subset V \text{ such that }  \pi(B) \leq L \bigr\} \]
 if $L\geq \min_{v \in V} \pi(v)$ and $\Lambda(L)=1$ otherwise.  Given $c>0$ and $\alpha\geq 0$, we say that a bounded degree graph satisfies \eqref{ass:SP} if 
\begin{equation}
\label{ass:SP}
\tag{$\mathrm{SP}_{\alpha,c}$}
\Lambda(x)^{-1} \leq \frac{1}{c} \log^\alpha \biggl[\frac{x}{\max_v \pi (v)} \biggr] \qquad \text{for every $x \geq 2 \max_v \pi(v)$.}
\end{equation}
(The normalization by the maximal degree has been included in order to simplify various calculations below.)

\begin{remark} 
\label{rem:nonstandardchoice}
We remark that our definition of the spectral profile is slightly non-standard. Indeed, when considering the \emph{continuous-time} random walk, one considers the smallest eigenvalue of $I_A-P_A$ rather than of $I_A-P_A^2$ as we do here. It turns out however that using $I_A-P_A^2$ is more natural in the discrete-time setting. A simple application of the Perron-Frobenius theorem shows that the two definitions differ by at most a factor of two. 
% This change is unproblematic for the following reason: 
% 
% which is equal to  $1-\la$, where $\la$ is the largest eigenvalue of $P_A$. We defined both  $\mathcal{E}_A$ and $\la(A)$ w.r.t.\ $I_A-P_A^2$.  Our definition is motivated by the fact that in discrete-time the operator  $I_A-P_A^{2}$ plays in the analysis of bounding the heat kernel via the spectral profile the role played by $I_A-P_A$ in continuous-time.  
% 
 % By the Perron-Frobenius Theorem,  $1-\la(A)$ is also the largest eigenvalue of $P_A$ in absolute value. Hence $\la(A)=1-\la^2=(1-\la)(1+\la) \in [1-\la,2(1-\la))$, and so the two definitions can vary by at most a factor 2. 
\end{remark}
% Let $\pi $ be the measure on $V$ which assigns each $v$ mass $\deg v$. We define the \textbf{spectral profile} of $G$ to be the function $\Lambda: \N \to (0,1]$ given by
 % \[ \Lambda(L):=\inf \{\la(B):B \subset V \text{ such that }  \pi(B) \leq L \} \]
 % for each $L\geq \min_{v \in V} \pi(v) $. 

%  We consider also the following ``heat kernel"  condition
% \begin{equation}
% \label{ass:HK}
% \tag{$\mathrm{HK}_{\gamma,c}$}
% \sup_v P^t(v,v)/\pi(v)  \leq \exp(-c t^{\gamma} ), \qquad \forall \, t \geq 0
% \end{equation}
% for some $c,\gamma>0$. Combining Corollaries \ref{cor:SP} and \ref{cor:anti-FK} below yields the following equivalence. \todo{We might want to comment that the next proposition is not new (even if it is not written anywhere, it might be considered folklore?!).}
% \begin{prop}
% \label{prop:equiv}
% There exist $c_1(c,\gamma,\max_v \pi (v))$ and $c_2(c,\alpha,\max_v \pi (v))$ such that  \eqref{ass:HK} implies   $(\mathrm{SP}_{\sfrac{1}{\gamma}-1,c_1})$ and   \eqref{ass:SP} implies   $(\mathrm{HK}_{\sfrac{1}{1+\alpha},c_2})$ for all $\alpha,\gamma \in [0,1]$ and $c>0 $.  
% \end{prop}
% \begin{remark}
% \label{rem:equiv}
% The function $\sfrac{1}{x}-1 $ is the inverse of $\sfrac{1}{1+x} $. Hence Proposition \ref{prop:equiv} asserts that for  $\gamma \in [0,1]$ condition  \eqref{ass:HK} holds for some $c>0 $  iff    $(\mathrm{SP}_{\sfrac{1}{\gamma}-1,c'})$ holds for some $c'>0 $.  
% \end{remark} 

\begin{remark}
\label{remark:isoperimetry}
Many readers will be more familiar with the \emph{isoperimetric profile} than with the spectral profile. We now briefly recall the relationship between these two profiles for their convenience; we will not apply the isoperimetric profile in the subsequent analysis. Let $G$ be an infinite, locally finite graph. Its isoperimetric profile $(\Phi_*(x))_{x\geq 1}$ is defined to be
\[
\Phi_*(x) = \inf\Bigl\{ \frac{1}{\pi(A)}\sum_{a\in A, b\in V \setminus A} \pi(a)P(a,b) : A \subset V,\, \pi(A) \leq x\Bigr\}.
\]
  A simple variation on Cheeger's inequality yields that
   % of the well-known discrete version of the Cheeger's inequality  of differential
% geometry
 \begin{equation}
\label{e:cheeger}
 \frac{1}{4}\Phi_{*}^2(x) \leq \Lambda(x) \leq \Phi_{*}(x)  \end{equation}   
 for every $x \geq \min_{v\in V} \pi(v)$; see the proof of \cite[Lemma 2.44]{MR2199053}. (Here we have a $1/4$ rather than the usual $1/2$ in the first inequality due to our nonstandard definition of $\Lambda$.)
\end{remark}

The next proposition states that if $\alpha=(1-\gamma)/\gamma$ then \eqref{ass:HK} and \eqref{ass:SP} are equivalent to within a controlled  change of the constant $c$.

\begin{prop}
\label{prop:equiv}
Let $0< \gamma \leq 1$ and let $\alpha=(1-\gamma)/\gamma$. Then for every $c>0$ and $M<\infty$ there exists $c_2=c_2(\gamma,c,M)$ such that the following hold for every connected, locally finite graph $G$ with maximum degree at most $M$:
\begin{enumerate}
	\item If $G$ satisfies 
 \eqref{ass:HK}, then $G$ satisfies $(\mathrm{SP}_{\alpha,c_2})$.
 \item
 If $G$ satisfies 
 \eqref{ass:SP}, then $G$ satisfies $(\mathrm{HK}_{\gamma,c_2})$.
    % \eqref{ass:SP} implies   $(\mathrm{HK}_{\sfrac{1}{1+\alpha},c_2})$ for all $\alpha,\gamma \in [0,1]$ and $c>0 $.  
\end{enumerate}
\end{prop}

\begin{proof}
The first item is a special case of \cite[Lemma 2.5]{MR3572333}, while the second item follows from \cite[Proposition II.1]{MR1418518}; see also \cite[Section 2]{doi:10.1093/imrn/rny034}.
\end{proof}

In light of this equivalence, it suffices to prove the following variation on \cref{cor:SP}.

\begin{prop}
\label{cor:SPb}
Let $G$ be an infinite, connected,  bounded degree graph satisfying \eqref{ass:SP} for some $\alpha \geq 0$ and $c>0$. Then there exists a  positive constant $c_1=c_1(\alpha,c)$ such that
\begin{equation}
\label{e:eSPcor}
\Pr_{\mu_D}[X_k \in D ] \leq   \left[\max_{u,v \in D } \frac{\pi(u)}{\pi(v)}\right]^{1/2} \exp\left[ - c_1 \min\left\{\frac{k}{\log^\alpha |D|},\, k^{1/(1+\alpha)}\right\} \right]\end{equation}
 for every finite set $D\subset V$ and every $t\geq 0$.
\end{prop}

\cref{cor:SPb} will in turn be deduced as a special case of the following proposition. We prove two variations on the same inequality: One of these bounds concerns random walk started at a uniform point of $D$, which is what arises in our analysis of percolation, while the other concerns random walk started at a point of $D$ chosen according to the probability measure $\pi_D(v) = \pi(D)^{-1} \pi(v)\mathbbm{1}(v\in D)$. This second bound is more natural from the random walk perspective, and we include it for future use since the proof is the same.

\begin{prop}
\label{prop:generalescape}
Let $G=(V,E)$ be a connected, locally finite graph with spectral profile $\Lambda$, and let $D \subseteq V$ be finite. If $\ell,k \geq 0$ satisfy
\begin{align}
\label{e:sp}
k &\geq \ell + 1  + \sum_{i=1}^{\ell} \frac{2 \log 4}{\Lambda \bigl(4^{i+1} \max_{v\in D} \pi(v)|D|\bigr)}   &\text{ then }&&  
 \Pr_{\mu_D}(X_k \in D ) &\leq \left[\max_{u,v\in D} \frac{\pi(u)}{\pi(v)}\right]^{1/2} 2^{-\ell}.
\intertext{Similarly, if $\ell,k\geq 0$ satisfy}
k &\geq \ell + 1  + \sum_{i=1}^{\ell} \frac{2 \log 4}{\Lambda \bigl(4^{i+1} \pi(D)\bigr)} 
 &\text{ then }&&  
 \Pr_{\pi_D}(X_k \in D ) &\leq  2^{-\ell}.
\end{align}
\end{prop}

We begin by introducing some basic notation. 
We identify each function $\phi \in \R^A$ with its extension to $\R^V$ obtained by setting $\phi \equiv 0$ on $V\setminus A$.  For $i > 0$ and $\phi \in \R^V$ let $P_A ^{i}\phi \in \R^A$ be given by 
\[P_A^{i}\phi (u):=\sum_v P_A^{i}(u,v)\phi(v)=\mathrm{E}_u\Bigl[\phi(X_i) \mathbbm{1}\bigl(T_{V\setminus A}>i\bigr)\Bigr],\]
where $T_{V\setminus A}=\inf\{k \geq 0 : X_k \in V \setminus A\}$ denotes the first time that the walk visits $V\setminus A$.   Similarly, for each signed measure $\mu$ on $V$ and  $i > 0$  let $\mu P_A^i $ be the signed  measure supported on $A$  given by 
\[
\mu P_A^i (u): =\sum_{v\in A} \mu(v) P_A^{i}(v,u)= \sum_{v \in A}\mu(v) \mathrm{P}_v\bigl(X_i=u, T_{V\setminus A}>i\bigr).
\]  
We also define $\langle \phi,\psi\rangle_\pi = \sum_{v\in V}\pi(v)\phi(v)\psi(v)$ for each $\phi,\psi\in \R^V$, and define $\|\phi\|_{2,\pi}^2 =\langle \phi,\phi\rangle_\pi$ and $\|\phi\|_{1,\pi}= \sum_{v\in V}\pi(v)|\phi(v)|$ for each $\phi\in \R^V$. Similarly, for each pair of signed measures $\mu,\nu$ on $V$ we define $\langle \mu,\nu\rangle_{1/\pi}= \sum_{v\in V}\mu(v)\nu(v)/\pi(v)$ and define $\|\mu\|_{2,1/\pi}^2 = \langle \mu,\mu\rangle_{1/\pi}$. The Dirichlet form $\cE_A:\R^V\to \R$ is defined by setting 
 \[\mathcal{E}_A(\phi):=\bigl\langle  (I_{A}-P_A^{2})\phi, \phi\bigr\rangle_\pi 
 % \text{ for each $\phi\in \R^A$}, \qquad \text{where} \qquad \langle \phi,\psi\rangle_\pi:=\sum_{v \in V}  \pi(v) \phi(v)\psi(v) .
 \] 
for every $\phi\in \R^V$. 
 % For $f \in \R^V $ let $\|f \|_{2}^2:=\langle f,f\rangle$ and  $\|f \|_{1}:=\sum_{v\in V} \pi(v) |f(v)|$. For signed measures $\mu$ and $\nu$ we define
% \[\langle  \mu ,\nu \rangle_{\pi}:=\sum_{v}\sfrac{1}{\pi (v)}\mu(v)\nu(v)=\langle  \sfrac{\mu}{ \pi},\sfrac{\nu}{\pi} \rangle \] and $\|  \mu \|_{2,\pi}^2:=\|\sfrac{\mu}{ \pi}\|_2^2 $, where $\sfrac{\mu}{ \pi} \in \R^V $ is given by $\sfrac{\mu}{ \pi}(v):=\sfrac{\mu (v)}{ \pi (v)}$. 
It is a standard fact that $\lambda(A)$ can be expressed alternatively in terms of the Dirichlet form as
\begin{equation}
\label{e:extremechar}
\la(A)=\inf \left\{ \frac{\EE_{A}(\phi)}{\| \phi\|_{2,\pi}^{2}} : \phi \in  \R_{+}^A,\, \phi \not \equiv 0 \right\}.
\end{equation}
Indeed, this follows from \cite[Theorem 3.33]{aldous-fill-2014}. 
% together with the central equality of \eqref{eq:removingthesquare}.

 We note that the reversibility of $P$ is inherited by $P_A^{k}$, so that for every $k \geq 0$ we have that  $\pi(u)P_A^k(u,v)=\pi(v)P_A^{k}(v,u) $ for every $u,v\in V$ and $k\geq 0$. This is easily seen to imply that
\begin{equation}
\label{e:selfadjoint}
\langle  \mu P_{A}^t ,\nu \rangle_{1/\pi}=\langle    P_{A}^t \sfrac{\mu}{ \pi},\sfrac{\nu}{ \pi}\rangle_\pi=\langle \sfrac{\mu}{ \pi},    P_{A}^t\sfrac{\nu}{ \pi}\rangle_\pi=\langle  \mu ,\nu  P_{A}^t \rangle_{1/\pi}
\end{equation} 
for every pair of signed measures $\mu$ and $\nu$, where $\frac{\mu}{\pi}$ and $\frac{\nu}{\pi}$ denote the functions $\frac{\mu}{\pi}(u)=\mu(u)/\pi(u)$ and $\frac{\nu}{\pi}(u)=\nu(u)/\pi(u)$ respectively.

The first step in the proof of \cref{prop:generalescape} is the following key lemma,  which is an analogue of  \cite[Lemma 2.1]{MR2199053}.

\begin{lemma}
\label{lem:key}
 For every non-zero  $\phi \in \R_+^A$ we have that
\[\frac{\EE_{A}(\phi)}{\|\phi \|_{2,\pi}^2} \geq \frac{1}{2} \Lambda \left( 4 \|\phi \|_{1,\pi}^2/\|\phi \|_{2,\pi}^2 \right).\]
\end{lemma}

\begin{proof}
Let $\beta:=\|\phi\|_{2,\pi}^2/4\|\phi\|_{1,\pi}$, and consider $B:=\{v \in A:\phi(v) \geq \beta   \} $. By H\"older's inequality and the fact that $\phi \geq 0$ we have that $\sup_v \phi(v) \geq \|\phi \|_{2,\pi}^2/\|\phi\|_{1,\pi}$, so that in particular the set $B$ is not empty. On the other hand, we clearly have that $ \pi( B) \leq \|\phi\|_{1,\pi}/\beta =4 \|\phi\|_{1,\pi}^2/\|\phi \|_{2,\pi}^2$. Defining the function $\psi:=(\phi-\beta  )\mathbbm{1}_{B}$, we have that
\[\|\psi \|_{2,\pi}^2 \geq \left(\|\phi \|_{2,\pi}^2-\|\phi^{2}\mathbbm{1}_{A \setminus B}  \|_{1,\pi}\right)-2  \beta\|\phi \mathbbm{1}_{B}  \|_{1,\pi} \geq \|\phi \|_{2,\pi}^2  -\beta \|u \mathbbm{1}_{A \setminus B}  \|_{1,\pi}-2  \beta\|\phi \mathbbm{1}_{B}  \|_{1,\pi} \geq \half  \|\phi\|_{2,\pi}^2,    \]
where we used $\phi^{2}\mathbbm{1}_{A \setminus B}\leq \beta \phi \mathbbm{1}_{A \setminus B}$ and $2\beta \|\phi\|_1=\sfrac{1}{2} \|\phi\|_{2,\pi}^2$. 
On the other hand, we have that
 \begin{align*}
 \EE_{B}(\psi)=\EE_{A}(\psi) &= \frac{1}{2} \sum_{u,v\in B} \pi(u) P_A^2(u,v) (\phi(u)-\phi(v))^2 + \sum_{u \in B, v \in A \setminus B} \pi(u) P_A^2(u,v) (\phi(u)-\beta)^2 \\
 &\leq \frac{1}{2} \sum_{u,v\in B} \pi(u) P_A^2(u,v) (\phi(u)-\phi(v))^2 + \sum_{u \in B, v \in A \setminus B} \pi(u) P_A^2(u,v) (\phi(u)-\phi(v))^2\\ &\leq \EE_{A}(\phi)
 \end{align*}
 and hence that
 \[\frac{\EE_{A}(\phi)}{\|\phi \|_{2,\pi}^2} \geq  \frac{1}{2}  \frac{ \EE_{B}(\psi)}{ \|\psi \|_{2,\pi}^2} \geq \frac{1}{2} \Lambda \Bigl( 4 \|\phi\|_{1,\pi}^2/\|\phi \|_{2,\pi}^2 \Bigr), \quad  \]
where in the second inequality we used \eqref{e:extremechar} and the fact that $\pi( B) \leq 4 \|\phi\|_{1,\pi}^2/\|\phi \|_{2,\pi}^2  $.
% , together with the standard fact (cf.\ \cite[Theorem 3.33]{aldous} combined with Remark \ref{rem:nonstandardchoice}) that for every finite $B$
\end{proof}

\begin{corollary}
\label{cor:spcor}
Let $\mu$ be a measure on $V$ with $\mu(V) \leq 1$. 
% Let $A \subset V $ be a finite set, let $\mu $ be a (non-negative) measure supported on $A$ with $\mu(A)\leq 1$, and let $\mu_k:=\mu P_{A}^{k}$. 
If $\ell,k\geq 0$ satisfy 
% Then  
\[
k \geq \ell + 1  + \sum_{i=1}^{\ell} \frac{2 \log 4}{\Lambda (4^{i+1}\|\mu \|_{2,1/\pi}^{-2})}  
\quad \text{ then } \quad 
\|\mu P^k \|_{2,1/\pi} \leq 2^{-\ell}\|\mu \|_{2,1/\pi}.
\]
\end{corollary}

\begin{proof} 
The claim holds vacuously if $\mu(V)=0$, so suppose not. 
Let $A \subset V$ be finite with $\mu(A)>0$, let $\mu_0:= \mu I_A$ be the restriction of $\mu$ to $A$, let $\mu_k := \mu_0 P_A^k$ for each $k\geq 1$, and 
 let $\phi_k:=P_A^k \sfrac{ \mu}{ \pi} \in \R_+^A $ for each $k\geq 0$. By \eqref{e:selfadjoint} we have that
\begin{equation}
\label{e:sp5}
\|\mu_{k} \|_{2,1/\pi}^2-\|\mu_{k+1} \|_{2,1/\pi}^2=\|\phi_{k} \|_{2,\pi}^2-\|\phi_{k+1} \|_{2,\pi}^2=\langle \phi_{k},\phi_{k}\rangle_\pi-\langle P_{A}^{2}\phi_{k},\phi_{k}\rangle_\pi  =\mathcal{E}_A(\phi_{k}). 
\end{equation}
Let $r_0=0$ and for each $\ell\geq 1$ let $r_\ell$ be maximal such that $\|\phi_{r_\ell}\|_{2,\pi}^2 > \|\phi\|_{2,\pi}^2/4^\ell$. 
% Since $ \|\phi_k \|_{1,\pi}=\mu_k(A) \leq 1 $ for all $k\geq 1$, 
% we deduce that if $\|u_{r} \|_{2}^2 > \|u \|_{2}^2/4$ then,
 Using the fact that the $L_2$ norm of $\phi_k$ is non-increasing in $k$, as well as   \eqref{e:sp5} and Lemma \ref{lem:key}, we deduce that
 \begin{multline*} \|\phi_{k+1} \|_{2,\pi}^2 
\leq \|\phi_{k} \|_{2,\pi}^{2}\left[1- \frac{1}{2}   \Lambda \left( 4 \|\phi_k\|_{1,\pi}^2/\|\phi_k \|_{2,\pi}^2 \right)\right]\\
 \leq \|\phi_{k} \|_{2,\pi}^{2}\left[1- \frac{1}{2}   \Lambda \left( 16/\|\phi_{r_{\ell-1}} \|_{2,\pi}^2 \right)\right]
 \leq \|\phi_{k} \|_{2,\pi}^{2}\left[1- \frac{1}{2}   \Lambda \left( 4^{\ell+1}/\|\phi_0 \|_{2,\pi}^2 \right)\right]
  \end{multline*}
 for every $\ell\geq 1$ and $r_{\ell-1} \leq k\leq r_{\ell}$, where we also used that $ \|\phi_k \|_{1,\pi}=\mu_k(A) \leq 1 $ for every $k\geq 1$ in the second inequality.
Thus, we have that
\[
\|\phi_{r_{\ell-1}+1}\|_{2,\pi}^2 4^{-1} \leq 
\|\phi\|_{2,\pi}^2 4^{-\ell} < \|\phi_{r_{\ell}}\|_{2,\pi}^2 \leq \|\phi_{r_{\ell-1}+1}\|_{2,\pi}^2\left[1- \frac{1}{2}   \Lambda \left( 4^{\ell+1}/\|\phi_0 \|_{2,\pi}^2 \right)\right]^{r_\ell-r_{\ell-1}-1}
\]
for every $\ell\geq 1$.
We deduce by an elementary calculation that $r_\ell-r_{\ell-1} -1 < (2 \log 4)/\Lambda( 4^{\ell+1}/\|\phi_0 \|_{2,\pi}^2 )$. It follows immediately that if $k,\ell \geq 0$ satisfy
\[
k \geq \ell + 1  + \sum_{i=1}^{\ell} \biggl\lfloor\frac{2 \log 4}{\Lambda (4^{i+1}\|\mu I_A \|_{2,1/\pi}^{-2})} \biggr\rfloor 
\quad \text{ then } \quad 
\|\mu P_A^k \|_{2,1/\pi} \leq 2^{-\ell}\|\mu I_A \|_{2,1/\pi}.
\]
The claim follows since the finite set $A$ was arbitrary. 
\end{proof}

We are now ready to prove \cref{prop:generalescape}.

% \emph{Proof of \eqref{e:sp}.}
\begin{proof}[Proof of \cref{prop:generalescape}]
 % It suffices to show that for every finite set $D \subset V$ and $k$,  for all finite $A \subset V $ we have that
 Let $\mu$ be a measure on $V$ with $\mu(V)\leq 1$, let $D \subseteq V$ be finite,  and suppose that $k,\ell \geq 0$ are such that
\begin{equation}
\label{e:klcondition}
% k \geq \sum_{i=1}^{\ell}\lceil \sfrac{2 \log 4}{\Lambda(\max_{v} \pi(v) |D|4^{i+1})}\rceil 
k \geq \ell + 1  + \sum_{i=1}^{\ell} \frac{2 \log 4}{\Lambda (4^{i+1}\|\mu \|_{2,1/\pi}^{-2})}.
% \quad\text{ then }\quad
% \Pr_{D}[X_t \in D,T_{V\setminus A}>t ]=
% \mu_DP_A^t(D)\leq \left[\max_{u,v \in D } \sfrac{\pi(u)}{\pi(v)}\right]  2^{-\ell}.
 % \quad \text{if} \quad  t \geq \sum_{i=1}^{\ell}\lceil \sfrac{2 \log 4}{\Lambda(\max_{v} \pi(v) |D|4^{i+1})}\rceil.
\end{equation}
% for every finite set $A \subseteq V$.
Observe that for any measure $\mu$ on $V$ we have by Cauchy-Schwarz that
% , applying Jensen's inequality to the distribution $\nu(v)=C^{-1}\pi(v)\mathbbm{1}(a\in D)$ yields that 
% \begin{equation}
% \label{e:sp3}
% \mu(D)^2 \leq 
% \max_{v\in D}\pi(v)^2 \biggl( \sum_{v \in D}\frac{\mu(v)}{\pi(v)}\biggr)^{2} 
% \leq 
% \max_{v\in D}\pi(v)^2
% \left[\sum_{v\in D} \pi(v)\right]
% \left[\sum_{v \in D} \frac{\mu(v)^{2}}{\pi(v)}\right]
% =
% \max_{v\in D}\pi(v)^2
 % \pi(D)\|\mu\|_{2,1/\pi}^2,
% \end{equation}
$\mu(D)^2 = \langle \pi \mathbbm{1}_D, \mu\rangle_{1/\pi}^2 \leq \pi(D) \|\mu\|_{2,1/\pi}^2$, and applying \cref{cor:spcor} we deduce that
\[
\mu P^k(D)^2 \leq \pi(D) \|\mu\|_{2,1/\pi}^24^{-\ell}
 % \left[\sum_{v\in D} \pi(v)\right] \|\mu\|_{2,1/\pi}^2 4^{-\ell}
\]
for every measure $\mu$ on $V$ with $\sum_{v\in V} \mu(v) \leq 1$. We conclude by applying this estimate to the uniform distribution $\mu_D$ on $D$ and the normalized stationary measure $\pi_D$ on $D$ and noting that 
% $\min_{v\in D} \pi(v)^{-1} |D|^{-1} \leq \|\mu_D\|_{2,1/\pi}^2 \leq \max_{v\in D} \pi(v)^{-1} |D|^{-1}$ 
$\|\mu_D\|_{2,1/\pi}^{-2} \leq [\max_{v\in D}\pi(v)] |D|$, that 
$\pi(D)\|\mu_D\|_{2,1/\pi}^2 \leq \max_{u,v\in D} \pi(u)/\pi(v)$, 
and that 
 $\|\pi_D\|_{2,1/\pi}^2= \pi(D)^{-1}$.
% Hence in light of \eqref{e:sp2} and the fact that  $C(D)\|\mu_D\|_{2,\pi}^2 = \sfrac{C(D)^{2}}{|D|^{2} } \leq \sfrac{1}{\min_{b \in D} \pi(b)^2} $, our task is to show that 
% \begin{equation}
% \label{e:sp4}
% \|\mu_D P_A^t \|_{2,\pi}\leq 2^{-\ell}\|\mu_D \|_{2,\pi} \qquad \text{if} \qquad t \geq \sum_{i=1}^{\ell} \lceil \sfrac{2 \log 4}{\Lambda(\max_{v  } \pi(v)|D|4^{i+1})} \rceil.
% \end{equation}
% Noting that $\|\mu_D \|_{2,\pi}^{-2}=\sfrac{|D|^2}{C(D)} \leq  \max_{v} \pi(v)|D| $, \eqref{e:sp4} follows from Corollary \ref{cor:spcor} below. 
\end{proof}

We now perform the calculation required to deduce \cref{cor:SPb} from \cref{prop:generalescape}.

\begin{proof}[Proof of \cref{cor:SPb}]
By \cref{prop:generalescape} and the assumption that $G$ satisfies \eqref{ass:SP}, we have that if $k,\ell\geq 0$ satisfy
\[
k \geq F(\ell) := \ell + 1 + \sum_{i=1}^\ell \frac{2 \log 4}{c} \log^\alpha\left[4^{i+1}|D|\right] \quad \text{ then } \quad \Pr_{\mu_D}(X_k \in D ) \leq \left[\max_{u,v\in D} \frac{\pi(u)}{\pi(v)}\right]^{1/2} 2^{-\ell}.
\]
Moreover, we clearly have that
\[
F(\ell) \leq (\ell+1)\left[1 + \frac{2 \log 4}{c} \log^\alpha(4^{\ell+1}|D|)\right] \leq C \max\bigl\{(\ell+1)^{1+\alpha},\, (\ell+1) \log^\alpha |D|\bigr\}
\]
for some constant $C=C(c,\alpha)$, and hence that if $k,\ell \geq 0$ satisfy
\[
k \geq C \max\bigl\{(\ell+1)^{1+\alpha},\, (\ell+1) \log^\alpha |D|\bigr\} \quad \text{ then } \quad \Pr_{\mu_D}(X_k \in D ) \leq  \left[\max_{u,v\in D} \frac{\pi(u)}{\pi(v)}\right]^{1/2} 2^{-\ell}.
\]
The result now follows by an elementary calculation.
\qedhere

% . Define $\tilde F(\ell) = C \max\bigl\{(\ell+1)^{1+\alpha},\, (\ell+1) \log^\alpha |D|\bigr\}$

% It suffices to consider $t \geq (2c \log 4) \log^\alpha (|D|4^{2})$, as small values of $t$ can then be handled by decreasing $c_1$ if necessary. Pick $\ell=\ell(t)$ such that $t \geq f(\ell):= (\sum_{i=1}^{\ell}(2c \log 4) \log^\alpha (|D|4^{i+1})) $ and $t < f(\ell+1)$. Then $f(\ell) \leq t< f(\ell+1) \leq 3 f(\ell) $ for all $\ell \geq 1$. 
% As $\half \leq \sfrac{(a+b)^{\alpha}}{a^{\alpha}+b^{\alpha}} \leq 2^{\alpha} $, under \eqref{ass:SP} (where $\alpha \in (0,1]$) we have that \[t/c\asymp \ell \log^{\alpha} |D|+ \sum_{i=1}^{\ell} i^{\alpha} \asymp  \ell  \max\{ \log^{\alpha} |D|,\sfrac{1}{\alpha} \ell^{\alpha}\}.    \]
% For $\ell \leq \alpha^{\sfrac{1}{\alpha}}  \log |D|  $ we get that $\ell \asymp \frac{t}{c\log^\alpha |D|} =\sfrac{1}{c} \min\left\{\frac{t}{\log^\alpha |D|},\, t^{1/(1+\alpha)}\right\} $, while for  $\ell \geq \alpha^{\sfrac{1}{\alpha}}  \log |D|  $ we get that $\ell \asymp   (\sfrac{\alpha t}{c})^{\sfrac{1}{1+\alpha}} \asymp (\sfrac{\alpha}{ c})^{\sfrac{1}{1+\alpha}} \min\left\{\frac{t}{\log^\alpha |D|},\, t^{1/(1+\alpha)}\right\} $. The claim  follows from \eqref{e:sp}.
\end{proof}

\begin{proof}[Proof of \cref{cor:SP}]
This follows immediately from \cref{cor:SPb,prop:equiv}.
\end{proof}

% The following offers a converse to \eqref{e:sp}. Namely, a lower bound on the return probability in terms of the spectral-profile. 

\section{Proofs of the main theorems}

In this section we deduce \cref{thm:main,thm:quanttransitive} from \cref{thm:twoghostquasitransitive,lem:surgery,cor:SP}. We first formulate a generalization of \cref{thm:quanttransitive} to the quasi-transitive case, which will then imply both \cref{thm:main,thm:quanttransitive}. The statement of this generalization will employ the following quantitative notion of quasi-transitivity. Let $G=(V,E)$ be a unimodular quasi-transitive graph, let $\cO$ be a complete set of orbit representatives for the action of $\Aut(G)$ on $V$, and let $\mu$ be as in \cref{sec:unimodular}. Given $r<\infty$ and $\eps>0$, we say that $G$ satisfies \eqref{ass:QT} if
\begin{multline*}
\label{ass:QT}
% \begin{array}{l}
\text{$\mu(o) \geq \eps$ for every $o\in \cO$, and for every $u,v \in V$}\\ \text{there exist $w\in [u]$, $z\in[v]$ such that $d(w,z)\leq r$.}
% \end{array} 
\tag{$\mathrm{QT}_{r,\eps}$}
\end{multline*}
Every unimodular transitive graph trivially satisfies \eqref{ass:QT} with $r=0$ and $\eps=1$, while every unimodular quasi-transitive graph satisfies \eqref{ass:QT} for \emph{some} $r<\infty$ and $\eps>0$, so that \cref{thm:main,thm:quanttransitive} both follow immediately from the following theorem.

\begin{theorem}
\label{thm:quantquasi}
Let $G=(V,E)$ be a unimodular quasi-transitive graph of maximum degree at most $M$ satisfying \eqref{ass:HK} and \eqref{ass:QT} for some $r<\infty$, $c,\eps>0$ and $\gamma >1/2$. Then for every $0\leq \beta <(2\gamma-1)/\gamma$ there exists a constant $K(\beta)=K(\beta,\gamma,c,r,\eps,M)$ such that
\[\bE_{p}\exp\left[\log^\beta |K_v|\right] \leq K(\beta)\]
for every $v\in V$ and $p\leq p_c$.
\end{theorem}

Given a unimodular quasi-transitive graph $G=(V,E)$, we let $\rho$ be a random root vertex of $G$ chosen according to the measure $\mu$, and write $\P_p$ and $\E_p$ for probabilities and expectations taken with respect to the joint law of $\omega_p$ and $\rho$. Recall that we also write $\mathrm{P}_v$ for the law of a simple random walk on $G$ started at the vertex $v$, and for each finite set $D\subset V$ we write $\mathrm{P}_{\mu_D}$ for the law of a simple random walk started from a uniform point of $D$. (The two uses of $\mu$ should not cause confusion.)
% We also write $a  \vee b:=\max\{a,b\}$ and $a \wedge b:=\min \{a,b\}$. 
% We define $K_u$ such that $u \in K_u$ and so $|K_u|$. 
% Below we define $\sfrac{1}{\log^{\alpha}1}=+  \infty $ for $\alpha>0$. 

\begin{lemma}
\label{lem:kappapk}
Let $G=(V,E)$ be a unimodular quasi-transitive graph with degrees bounded by $M$ satisfying \eqref{ass:HK} for some $c>0$ and $0<\gamma \leq 1$, and let $\alpha=(1-\gamma)/\gamma$. Then for every $\beta\in (0,1]$  there exists a constant $c_2(\beta)=c_2(\beta,\alpha,c,M) $, such that  
\[ \E_p\left[\mathrm{P}_{\rho}(X_k \in K_\rho)\right] \leq 2\left[\max_{u,v\in V}\frac{\deg(u)}{\deg(v)}\right]^{1/2}  \E_p \exp\left[\log^\beta |K_{\rho}|\right] \exp\left[-c_2(\beta) k^{\beta/(\alpha+\beta)}\right].
\]
for every $k\geq 1$ and $0\leq p \leq 1$.
 % where $(X_t)_{t \geq 0 } $ is a  SRW, independent of $\omega_p $, and $\rho$ is random root chosen from a collections of representatives $\mathcal{O}$ of the orbits of $V$ under $\mathrm{Aut}(G)$ according to the distribution $\mu$ w.r.t.\ which $(G,\rho) $ is unimodular. \todo{can be shorten to ``$\rho$ is a random root as in \S blabla"} 
\end{lemma}

\begin{proof} 
The claim is trivial if $p$ is such that $|K_\rho|=\infty$ with positive probability, so suppose not.
Let $c_1=c_1(\gamma,c,M)$ be the constant from \cref{cor:SP}.
% The first inequality is obvious. 
Applying the mass-transport principle to the function $f:V^2\to [0,\infty]$ defined by 
\[f(u,v)=\bE_p\left[\frac{\mathbbm{1}(v\in K_u)}{|K_u|}\Pr_u[X_k \in K_u] \right], \]
we deduce that
\[
\E_p\left[\mathrm{P}_{\rho}(X_k \in K_\rho)\right] = \E\sum_{v\in V}f(\rho,v) = \E\sum_{v\in V}f(v,\rho) = \E_p\left[\Pr_{\mu_{K_\rho}}(X_k \in K_\rho)\right],
\]
and we deduce from \cref{cor:SP} that
\begin{multline}
\label{eq:421}
\E_p\left[\mathrm{P}_{\rho}(X_k \in K_\rho)\right] \leq 
\left[\max_{u,v\in V}\frac{\deg(u)}{\deg(v)}\right]^{1/2} \E_p \exp\left[ -c_1 \min\left\{\frac{k}{\log^\alpha|K_\rho|},k^\gamma\right\}\right]
\\\leq 
\left[\max_{u,v\in V}\frac{\deg(u)}{\deg(v)}\right]^{1/2} \left[\E_p \exp\left[ -c_1 \frac{k}{\log^\alpha|K_\rho|}\right] + e^{-c_1 k^\gamma}\right].
\end{multline}
%  and \eqref{e:eSPcor} 
% \begin{equation}
% \begin{split}
% \label{e:3.1}
%  \E_p\left[\P_\rho(X_k \in K_{\rho}) \right] & = \E_p\left[\P_{K_\rho}(X_k \in K_\rho)\right] \leq ( \sfrac{\max_{a \in V }\pi(a)}{\min_{a \in V }\pi(b)})  \E_p [L(k,K_\rho)], \quad \text{where}
% \\ L(k,K_\rho) &  \leq \exp\left[-c_1 k^{1/(1+\alpha)}\right] \vee \exp\left[ - c_1 \sfrac{k}{\log^\alpha |K_\rho|} \right] .
% \end{split}
% \end{equation}
Using the inequality $\E_p[g(|K_\rho|)] \leq \sup_{x \geq 1 }[g(x)/h(x)]\E_p[h(|K_\rho|)] $, which holds for every $g,h:\N \to \R_+$, we deduce that 
\[
\E_p\exp\left[ - c_1 \frac{k}{\log^\alpha |K_\rho|} \right] \leq \sup_{x\geq 1}\exp\left[-\log^\beta x - c_1\frac{k}{\log^\alpha x}\right]\E_p\exp\left[ \log^\beta |K_\rho|\right]. 
\]
% Since $\log^\beta x$ is an increasing function of $x$ and $c_1 k \log^{-\alpha}x$ is a decreasing function of $x$, \todo{I don't mind omitting the sentence about the rule of thumb} a rule of thumb is the supremum is can be approximated by picking $x$ for which   the two terms in the exponent are equal, which occurs when $\log^{\alpha+\beta}x = c_1k$. 
A direct and elementary calculation shows that the minimum of $\log^\beta x + c_1k  \log^{-\alpha} x$ is attained when $\log^{\alpha+\beta}x =\alpha c_1k/\beta $, and we deduce that there exists a constant $c=c(\alpha,\beta,c_1)$ such that
\begin{equation}
\label{eq:422}
\E_p\exp\left[ - c_1 \frac{k}{\log^\alpha |K_\rho|} \right] \leq \E_p\exp\left[ \log^\beta |K_\rho|\right] \exp\left[-c k^{\beta/(\alpha+\beta)}\right].
\end{equation}
Taking $c_2 = \min\{c_1,c\}$, the proof is now easily concluded by combining \eqref{eq:421} and \eqref{eq:422} and noting that $\gamma = 1/(\alpha+1) \geq \beta/(\alpha+\beta)$ and that $\E_p\exp\left[ \log^\beta |K_\rho|\right] \geq 1$.
%  Writing $J(k)=J(\alpha,\beta,k):=\exp\left[-c_{2} k^{\beta/(\alpha+\beta)}\right]$ where $c_2:=c_{1}((\sfrac{\beta}{\alpha c_{2}})^{\alpha/(\alpha+\beta)} \wedge 1)$   we get that
% \[
% \E_p\exp\left[ - c_1 \sfrac{k}{\log^\alpha |K_\rho|} \right] \leq J(k)\E_p\exp\left[ \log^\beta |K_\rho|\right].
% \]
% Putting everything together, (using $\exp\left[-c_1k^{1/(1+\alpha)}\right] \leq J(k) $ as $\beta < 1$) we get that
% \[
% \E_p [L(k,K_\rho)]\leq   \exp\left[-c_2k^{1/(1+\alpha)}\right] \vee J(k)\E_p\exp\left[ \log^\beta |K_\rho|\right]  
% \leq (1\vee \E_p\exp\left[ \log^\beta |K_\rho|\right]  
% )J(k).
% \]
% Substituting this in \eqref{e:3.1} concludes the proof.
\end{proof}

\begin{proof}[Proof of \cref{thm:quantquasi}] Let $\alpha=(1-\gamma)/\gamma$ and let $ 0 < \beta <(2\gamma-1)/\gamma = 1-\alpha $. Note that such a $\beta$ exists precisely when $\gamma>1/2$. Recall that $M$ is a constant satisfying $\max_{v\in V}\deg(v) \leq M$. 
\cref{thm:twoghostquasitransitive} immediately implies that there exists a constant $C=C(M,r,\eps)$ such that
\begin{equation}
\label{eq:convenienttwoghost}
\sup_{e\in E}\bP_p(\mathscr{S}_{e,n}) \leq C \left[\frac{1-p}{pn}\right]^{1/2}.
\end{equation}
for every $p\in [0,1]$ and $n\geq 1$, where the event $\mathscr{S}_{e,n}$ is defined as in the introduction. Recall that we define $P_p(n) = \inf_{v\in V}\bP_p(|E(K_v)|\geq n)$ and $\kappa_p(k)=\inf\{\tau_p(u,v) : d(u,v)\leq k\}$ for each $p\in [0,1]$ and $n,k\geq 1$. Note also that we have the elementary bound $p_c \geq 1/(\max_{v\in V}\deg(v) -1) > 1/2M$.

We have trivially that $\kappa_p(k) \leq \E_p[\mathrm{P}_\rho(X_k \in K_\rho)]$ for every $0\leq p \leq 1$ and $k\geq 1$. Thus, applying \cref{lem:surgery,lem:kappapk} and rearranging, we obtain that
\begin{align*}
P_p(n)^2 &\leq C \left[ \sum_{i=0}^{k-1} p^{-i}\right] \left[\frac{1-p}{pn}\right]^{1/2} + \kappa_p(k) \\&\leq 
 C \left[ \sum_{i=0}^{k-1} p^{-i}\right] \left[\frac{1-p}{pn}\right]^{1/2} + 2\left[\max_{u,v\in V} \frac{\deg(u)}{\deg(v)}\right]^{1/2}\E_p\exp\left[\log^\beta|K_\rho|\right]\exp\left[-c_2(\beta)k^{\beta/(\alpha+\beta)}\right]
\end{align*}
for every $0\leq p <p_c$ and $k,n\geq 1$.
Taking $k=k_{n}= \lceil \tfrac{1}{4}\log_{1/2M} n \rceil$ and using that $\sum_{i=0}^{k-1} p^{-i} \leq 2(2M)^k$ for $1/2M\leq p \leq 1$, we deduce by elementary calculation that there exist positive constants $C_1,C_2$ and $c_3$ depending only on $c,\alpha,\beta,r,\eps,$ and $M$ such that
\begin{multline*}
P_p(n)^2 \leq C_1  n^{-1/4} + C_1 \E_p\exp\left[\log^\beta|K_\rho|\right]\exp\left[-c_3 \log^{\beta/(\alpha+\beta)}n\right]\\
\leq C_2 \E_p\exp\left[\log^\beta|K_\rho|\right]\exp\left[-c_3 \log^{\beta/(\alpha+\beta)}n\right]
\end{multline*}
for every $1/2M \leq p <p_c$ and $n\geq 1$. On the other hand, for each $u,v \in V$ and $n\geq 1$ we have by the Harris-FKG inequality that
\begin{multline*}
% \bP_p(|K_v|\geq n) \leq \bP_p(|E(K_v)|\geq n) \leq 
\bP_p(|E(K_u)|\geq n) \geq \bP_p\bigl(\{|E(K_v)| \geq n\} \cap \{u \leftrightarrow v\}\bigr) \\\geq \tau_p(u,v) \bP_p(|E(K_u)|\geq n) \geq p^{d(u,v)}\bP_p(|E(K_u)|\geq n),
\end{multline*}
and we deduce that there exists a constant $C_3=C_3(M,r)$ such that 
\[
\sup_{v\in V}\bP_p(|K_v|\geq n) \leq \sup_{v\in V}\bP_p(|E(K_v)|\geq n) \leq C_3 P_p(n)
\]
for every $1/2M\leq p \leq 1$ and $n\geq 1$. Putting this all together, it follows that there exists a constant $C_4=C_4(c,\alpha,\beta,r,\eps,M)$ such that
\begin{equation}
\sup_{v\in V}\bP_p\bigl(|K_v|\geq n\bigr) \leq C_4 \exp\left[-\frac{c_3}{2}\log^{\beta/(\alpha+\beta)}n\right] \sqrt{\E_p \exp \left[\log^\beta|K_\rho|\right]} 
\end{equation}
for every $1/2M\leq p <p_c$ and $n\geq 1$, and hence that
\begin{align*}
\sup_{v\in V} \bP_p\Bigl(\exp\left[\log^\beta|K_v|\right] \geq x \Bigr) &\leq 
\sup_{v\in V}\bP_p\Bigl(|K_v|\geq \exp\left[ \log^{1/\beta} x\right] \Bigr)\\
&\leq 
C_4  \exp\left[-\frac{c_3}{2}\log^{1/(\alpha+\beta)} x\right] \sqrt{\E_p \exp \left[\log^\beta|K_\rho|\right]}
\end{align*}
for every $1/2M\leq p <p_c$ and $x\geq 1$. We integrate this bound to obtain that, since $|K_v|\geq 1$, 
\begin{multline*}
\sup_{v\in V} \bE_p\left[\log^\beta|K_v|\right] \leq \int_1^\infty 
\sup_{v\in V} \bP_p\Bigl(\exp\left[\log^\beta|K_v|\right] \geq x \Bigr) \dif x \\
\leq 
C_4 \sqrt{\E_p\exp \left[\log^\beta|K_\rho|\right]} \int_1^\infty \exp\left[-\frac{c_3}{2}\log^{1/(\alpha+\beta)} x\right] \dif x
\end{multline*}
for every $1/2M \leq p <p_c$. 
Since $\alpha+\beta<1$ this integral converges, and we obtain that there exists a positive constant $C_5=C_5(c,\alpha,\beta,r,\eps,M)$ such that
\begin{equation}
\label{eq:done}
\E_p\exp\left[\log^\beta|K_\rho|\right] \leq 
\sup_{v\in V} \bE_p\exp\left[\log^\beta|K_v|\right] \leq 
C_5 \sqrt{\E_p \exp \left[\log^\beta|K_\rho|\right]}
\end{equation}
for every $1/2M \leq p <p_c$. 
If $p<p_c$ then we have by sharpness of the phase transition \cite{MR852458,aizenman1987sharpness,duminil2015new} that $\E_p |K_\rho| <\infty$ and consequently that $\E_p \exp\left[ \log^\beta|K_\rho| \right]\leq \E_p |K_\rho| <\infty$. Thus, we may safely rearrange \eqref{eq:done} and deduce that there exists a constant $C_6=C_6(c,\alpha,\beta,r,\eps,M)$ such that 
\[
\E_p \exp \left[\log^\beta|K_\rho|\right] \leq \sup_{v\in V}\bE_p \exp \left[\log^\beta|K_v|\right] \leq C_6
\]
for every $1/2M \leq p < p_c$. Coupling $\omega_p$ for different values of $p$ in the standard monotone fashion (see e.g.\ \cite[Page 11]{grimmett2010percolation}) and applying the monotone convergence theorem implies that this bound continues to hold at $p_c$, completing the proof.
 \qedhere

\end{proof}

\begin{proof}[Proof of \cref{thm:main,thm:quanttransitive}]
Both results are immediate consequences of \cref{thm:quantquasi}.
\end{proof}

\section{Closing remarks}
\label{sec:closing}

\begin{remark}
\label{remark:speed}
Suppose that $G$ is a quasi-transitive graph satisfying \eqref{ass:HK} for some $c>0$  and $\gamma>0$, and suppose that the simple random walk on $G$ satisfies a bound of the form $\P( d(X_0,X_n) \leq Cn^\nu) \geq c$ for some $1/2\leq \nu \leq 1$, $c>0$ and $C<\infty$. (A theorem of Lee and Peres \cite{MR3127886} implies that such an inequality cannot hold for $\nu<1/2$.) Then the proof of \cref{thm:main} can easily be generalized to show that critical percolation on $G$ has no infinite clusters under the assumption that $(1-\gamma)\nu<\gamma$.  We have not included the proof of this stronger result since we do not know of any examples that we can prove satisfy this condition but do not satisfy the hypotheses of \cref{thm:main}. However, Tianyi Zheng has informed us that this stronger theorem \emph{might} apply to Cayley graphs of the first Grigorchuk group, for which the optimal values of $\gamma$ and $\nu$ are unknown.
\end{remark}

\begin{remark}
More generally, a similar analysis to that discussed in \cref{remark:speed} shows the following: Suppose that $G$ is a  quasi-transitive graph for which  there exists a symmetric stochastic matrix on $G$ that is invariant under the diagonal action of $\Aut(G)$ and for which the associated random walk $X$ satisfies \eqref{ass:HK} for some $c>0$ and $\gamma>0$ and satisfies $\P( d(X_0,X_n) \leq Cn^\nu) \geq c$ for some $1/2\leq \nu <\infty$, $c>0$ and $C<\infty$. (One may need to take $\nu>1$ if the walk takes long jumps.) If $(1-\gamma)\nu<\gamma$ then critical percolation on $G$ has no infinite clusters almost surely. Long-range random walks have been a powerful tool for analyzing specific examples of groups of intermediate growth, see e.g.\ \cite{erschler2018growth}.
\end{remark}

\begin{remark}
% The proof of \cref{thm:main} can easily be extended to, for example, quasi-transitive graphs satisfying a return probability bound of the form $p_n(v,v) \leq \exp\bigl[- \Omega(n^{1/2} \log^a n)\bigr]$ for $a>0$.
 We expect that with a sufficiently delicate analysis one can push our method to handle all quasi-transitive graphs satisfying a return probability bound of the form $p_n(v,v) \leq \exp\bigl[- \omega(n^{1/2})\bigr]$, as well as all quasi-transitive graphs satisfying $p_n(v,v) \leq \exp\bigl[- \Omega(n^{1/2})\bigr]$ and for which the random walk has zero speed in the sense that $d(0,X_n)/n \to 0$ a.s.\ as $n\to\infty$. 
Since \cref{conj:pc} is already known in the exponential growth case and the random walk on any graph of subexponential growth has zero speed, this would allow one to extend \cref{thm:main} to the case $\gamma=1/2$. (The fact that the random walk on a graph of subexponential growth has zero speed is an immediate consequence of the Varopoulos-Carne bound, see \cite[Theorem 13.4]{LP:book}.)
On the other hand, it seems that a new idea is needed to handle the case $\gamma<1/2$, and a solution to the following problem would be a promising next step towards \cref{conj:pc}.

% One can even dream that some of the same ideas we use here might appear in a future proof that \cref{conj:pc} holds in full generality. For the time being, a solution to the following problem would be a promising next step.
\end{remark}

\begin{problem}
Extend \cref{thm:main} to quasi-transitive graphs satisfying a return probability estimate of the form $p_n(v,v) \leq \exp\left[-\Omega(n^\gamma)\right]$ for some $0<\gamma<1/2$.
\end{problem}

The methods of \cite{doi:10.1093/imrn/rny034} may be relevant. 
Note that if Grigorchuk's \emph{gap conjecture} \cite{MR3174281} is true, then a solution to this problem for all $0<\gamma<1/2$ would settle \cref{conj:pc} for all Cayley graphs of intermediate growth. Indeed, if the strong version of the conjecture is true then it would suffice to consider the case $\gamma \geq 1/5$.

\begin{remark}
\label{remark:obstacle}
We now indicate why one should not expect our proof to handle the case $\gamma <1/2$ without substantial modification.
% Let us now see why  the strage
 Let $G$ be transitive and satisfy \eqref{ass:HK} for some $0<\gamma \leq 1$ and $c_1>0$. As before, we let $\alpha = (1-\gamma)/\gamma$, so that $\alpha<1$ if and only if $\gamma >1/2$. Our proof depends upon the analytic consequences of the inequality
\begin{align}
\bP_p(|K| \geq n)^2 &\leq \min_{k \geq 1} \left[\bE_p\exp\left[-c_2 \min \left\{k^\gamma, \frac{k}{\log^\alpha |K| }\right\}\right] + C_1 p^{-k} n^{-1/2}\right]
\label{eq:pair1}
% \\
% \]
% and
% \[
% \kappa_p(k) &\leq .
% \label{eq:pair2}
\end{align}
for appropriate constants $C_1 \geq 1$ and $c_2>0$. 
Let us suppose that we had the even stronger inequality
\[
\bP_p(|K| \geq n)^2 \leq \min_{x \geq 0} \left[\bE_p\exp\left[-\frac{c_2x}{\log^\alpha |K|}\right] + C_1 p^{-x} n^{-1/2}\right].
\]
We certianly never want to take $x \geq \log_{1/p} n$ when realizing the above minimum, and since the expectation on the right hand side is decreasing in $x$, this inequality is weaker than the inequality
\[
\bP_p(|K| \geq n)^2 \leq  \bE_p\exp\left[- \frac{c_3 \log n}{\log^\alpha |K|}\right] 
\]
for appropriate choice of $c_3>0$.
% \[
% \E_p\exp\left[- \frac{c_3 \log n}{\log^\alpha |K|}\right]  \geq C \P(\log^\alpha |K| \geq \log n )
% \]
Since $ \bE_p\exp\left[- \frac{c_3 \log n}{\log^\alpha |K|}\right]  \geq e^{-c_3}  \bP_p\bigl(|K| \geq e^{ \log^{1/\alpha} n} \bigr)$, this inequality is, in turn, weaker than the inequality
\begin{equation}
\bP_p(|K| \geq n)^2 \leq e^{-c_3} \bP_p\Bigl(|K| \geq e^{ \log^{1/\alpha} n} \Bigr).
 % \E_p\exp\left[- \frac{c_3 \log n}{\log^\alpha |K|}\right] 
 \label{eq:strong}
\end{equation}
Thus, any analytic consequence of the inequality \eqref{eq:pair1} must also be a consequence of the inequality \eqref{eq:strong}.
If $\alpha \geq 1$ then $e^{\log^{1/\alpha} x} \leq x$ for every $x \geq 1$, so that any decreasing function $f:[1,\infty] \to [0,1]$ with $f(x) \leq e^{-c_3}$ for all $x \geq 1$ trivially satisfies the inequality $f(x)^2 \leq e^{-c_3} f\bigl( e^{ \log^{1/\alpha} x} \bigr)$. Thus, the inequality \eqref{eq:strong} does not yield any non-trivial information on the rate of decay of $\bP_p(|K| \geq n)$ in this case. (The reader may find it illuminating to consider the constraints that \eqref{eq:strong} would place on the rate of decay of $\bP_p(|K| \geq n)$ when $\alpha <1$.) This appears to present a serious obstruction to extending our method to the case $\gamma < 1/2$. 
\end{remark}

\subsection*{Acknowledgments}
We thank Gidi Amir and Tianyi Zheng for sharing their expertise on groups of intermediate growth. JH was supported financially by 
the EPSRC grant EP/L018896/1.

 \setstretch{1}
 \footnotesize{
  \bibliographystyle{abbrv}
  \bibliography{unimodularthesis.bib}
}
\end{document}